\documentclass[]{mcom-l}

\usepackage{latexsym,amsfonts,amsmath,amssymb}

\newtheorem{theorem}{Theorem}[section]
\newtheorem{lemma}[theorem]{Lemma}

\theoremstyle{definition}
\newtheorem{definition}[theorem]{Definition}

\newtheorem{algorithm}{Algorithm}

\theoremstyle{remark}

\numberwithin{equation}{section}

\newcommand{\norm}[2][{}]{\ensuremath{\left \lVert #2 \right \rVert}_{#1}}
\DeclareMathOperator*{\argmin}{argmin}
\newcommand{\trc}{{\rm trc}}

\newcommand{\bsd}{\boldsymbol{d}}

\newcommand{\bsx}{\boldsymbol{x}}

\newcommand{\bsy}{\boldsymbol{y}}

\newcommand{\bsl}{\boldsymbol{l}}
\newcommand{\bsk}{\boldsymbol{k}}

\newcommand{\bszero}{\boldsymbol{0}}
\newcommand{\HH}{\mathcal{H}}
\newcommand{\LL}{\mathcal{L}}

\newcommand{\FF}{\mathbb{Z}}

\newcommand{\de}{\mathrm{e}}
\newcommand{\icomp}{\mathrm{i}}
\newcommand{\rd}{\mathrm{d}}

\newcommand{\nat}{\mathbb{N}}

\newcommand{\complex}{\mathbb{C}}
\newcommand{\NN}{\mathbb{N}}
\newcommand{\DD}{\mathcal{D}}
\newcommand{\KK}{\mathcal{K}}
\newcommand{\CC}{\mathcal{C}}

\newcommand{\wal}{\mbox{wal}}
\newcommand{\xx}{\mathbf{x}}

\newcommand{\landau}{\mathcal{O}}
\newcommand{\Order}{\mathcal{O}}

\begin{document}

\title[Discrete Walsh Transform]{A Multivariate Fast Discrete Walsh Transform with an Application to Function Interpolation}

\author{Kwong-Ip Liu}
\address{Department of Mathematics, Hong Kong Baptist University, Kowloon Tong, Hong Kong, P.\ R.\ China}
\curraddr{}
\email{kiliu@math.hkbu.edu.hk}

\author{Josef Dick}
\address{School of Mathematics and Statistics, University of New South Wales, Sydney 2052, Australia}
\curraddr{}
\email{josi@maths.unsw.edu.au}

\author{Fred J. Hickernell}
\address{Department of Applied Mathematics, Illinois Institute of Technology, Room E1-208, 10 W. 32nd Street, Chicago, IL 60616, USA}
\curraddr{}
\email{hickernell@iit.edu}
\thanks{This research was supported in part by the Hong Kong Research Grants Council grant HKBU/2009/04P, the HKSAR Research Grants Council Project No.\ HKBU200605 and the United States National Science Foundation grant NSF-DMS-0713848.}

\subjclass[2000]{42C10,41A15}
\date{\today}

\dedicatory{}

\begin{abstract}
For high dimensional problems, such as approximation and integration, one cannot afford to sample on a grid because of the curse of dimensionality.  An attractive alternative is to sample on a low discrepancy set, such as an integration lattice or a digital net. This article introduces a multivariate
fast discrete Walsh transform for data sampled on a digital net that
requires only $\Order(N \log N)$ operations, where $N$ is the number
of data points. This algorithm and its inverse are digital analogs of multivariate fast Fourier transforms. 

This fast discrete Walsh transform and its inverse may be used to
approximate the Walsh coefficients of a function and then construct a spline interpolant of the function.  This interpolant may then be used to estimate the function's effective dimension, an important concept in the theory of numerical multivariate integration. Numerical results for various functions are presented.
\end{abstract}

\maketitle

\section{Introduction}

The idea of the fast Fourier transforms goes back to Gauss and has been
popular ever since the seminal work of Cooley and Tukey \cite{CT} .
Let $f$ be a function from $[0,1]$ to the complex numbers. The task
is to compute
$$
\widetilde{f}(k) = \sum_{n=0}^{N-1} f(n/N) \de^{2\pi \icomp k n/N}
\quad
\mbox{ for } k = 0,\ldots, N-1.
$$ 
These are $N$ sums, each
consisting of $N$ summands. Hence a straight forward calculation
would have complexity of $\landau(N^2)$ operations. But the sums
have a certain structure which can be exploited. Indeed Cooley and Tukey showed that those sums can be
computed with $\landau(N \log N)$ operations. (There is some
dependence of the implied constant on the number $N$; the algorithm works best if $N$ is a prime power, see \cite{CT}).

In higher dimensions an effect commonly referred to as
{\it the curse of dimensionality} occurs. Let $f : [0,1]^s
\rightarrow \complex$ and consider the discrete Fourier transform
$$
\widetilde{f}(\bsk) =
\sum_{n_1,\ldots,n_s = 0}^{p-1} f(n_1/p,\ldots, n_s/p) \de^{2\pi
\icomp (n_1 k_1 + \cdots + n_s k_s)/p}
$$ 
for all $\bsk =(k_1,\ldots, k_s) \in \{0,\ldots, p-1\}^s$. Here the number of
points sampled is $N = p^s$. Hence if $s$ is large, say $100$ for example,
then choosing even $p = 2$ yields $N=2^{100} \approx 10^{30}$ points, making
such a computation infeasible for contemporary computers.

In the example above the design or set of sample points is a grid aligned with the coordinate axes, $\{(n_1/p,\ldots,n_s/p): 0 \le n_j < p\}$. To avoid the curse of dimensionality one needs a much smaller point set that is constructed differently. Such point sets have previously been considered in the context
of numerical integration, see \cite{N92,SJ94}. Two popular
construction methods are integration lattices (see \cite{SJ94}) and digital
nets and sequences (see \cite{N92}). This article focuses
on the latter family of points.  (Numerical approximation using lattice rule designs and an FFT has been treated
in \cite{LH,ZengLeungHickernell2004}.) The first examples of digital sequences were given by
Sobol~\cite{sob67} and Faure~\cite{fau82} before Niederreiter
introduced the general concept of $(t,m,s)$-nets and
$(t,s)$-sequences.  See \cite{nie05} for a recent survey. These
constructions yield extremely well distributed point sets if the
quality parameter $t$ is small. Digital $(t,m,s)$-nets are a special
construction of $(t,m,s)$-nets, and in the same way, digital
$(t,s)$-sequences are a special construction of $(t,s)$-sequences.  Digital constructions are introduced below.

\begin{definition}
Let $\FF_p$ be a finite field of prime order $p$, let
$C_1,\ldots,C_s$ be $s$ $m\times m$-matrices over $\FF_p =
\{0,1,\ldots, p-1\}$. The digital $(t,m,s)$-net $P(\CC) =
\{\bsx_0,\ldots, \bsx_{p^m-1} \}$, based on $\CC = (C_1,\ldots,C_s)$,
is then defined as follows: let $0 \le n < p^m$ and $n = n_0 + n_1 p +
\cdots + n_{m-1} p^{m-1}$ be the base $p$ representation of $n$.
Define $\vec{n} = (n_0,\ldots, n_{m-1})^T \in \FF_p^m$ and let
\begin{equation}
\label{defy} \vec{y}_{j,n} = C_j \vec{n} \in
\FF_p^m.
\end{equation}
Express $\vec{y}_{j,n}$ as $(y_{j,n,1},\ldots, y_{j,n,m})^T \in \FF_p^m$, and
then define
$$x_{j,n} = y_{j,n,1} p^{-1} + \cdots + y_{j,n,m} p^{-m}.$$ The
$n$-th point $\bsx_n$ of the digital net $P(\CC)$ over the finite
field $\FF_p$ is given by $\bsx_n = (x_{1,n},\ldots,
x_{s,n})$.
\end{definition}

The $t$ value is a non-negative integer such that for all $0 \le
d_1,\ldots, d_s \le m-t$ with $d_1 + \cdots + d_s = m-t$ the
system of vectors $c^{(1)}_1,\ldots, c^{(1)}_{d_1}, \ldots,
c^{(s)}_1,\ldots, c^{(s)}_{d_s}$ is linearly independent over
$\FF_p$. Here $c^{(j)}_k$ refers to the $k$-th row of the matrix
$C_j$. For a geometrical interpretation of the $t$ value see, for
example, \cite{N92}.  Smaller values of $t$ characterize more uniformly distributed nets.

Digital nets are often used in conjunction with certain wavelets,
namely Haar functions, first used by Sobol~\cite{sob69},
and Walsh functions, first used by Larcher~\cite{L93} and Larcher
and Traunfellner~\cite{LT94}.

Next, Walsh functions, which are piecewise constant, in base $p$ are defined. For more information on Walsh functions see for example
\cite{chrest, walsh} (or in the context of numerical integration see \cite{DP03a}). Throughout this article let $\nat_0$
denote the set of non-negative integers and let $\NN$ denote the set
of positive integers.
\begin{definition}
Let $p \ge 2$ be an integer. For a non-negative integer wavenumber $k$ with
base $p$ representation
\[
    k = k_0 + k_1 p + \cdots + k_{a-1} p^{a-1},
\]
with $k_i \in \FF_p$, the Walsh function $_p\wal_{k}:[0,1)
\longrightarrow \complex$ is defined by
\[
   _p\wal_{k}(x) := \omega_p^{x_1 k_0 + \cdots + x_a
k_{a-1}},
\] 
where $\omega_p = \de^{2\pi \icomp/p}$,
for $x \in [0,1)$ with base $p$ representation $x =
x_1/p+x_2/p^2+\cdots $ (unique in the sense that
infinitely many of the $x_i$ must be different from $p-1$). If it
is clear which base $p$ is chosen we simply write $\wal_{k}$.
\end{definition}

\begin{definition}
For dimension $s \geq 2$, $x_1, \ldots, x_s \in [0,1)$ and $k_1,
\ldots, k_s \in \nat_0$ define $_p\wal_{k_1,\ldots, k_s} : [0,1)^s
\longrightarrow \complex$ by
\[
    _p\wal_{k_1,\ldots,k_s}(x_1,\ldots,x_s) := \prod_{j=1}^s\,
_p\wal_{k_j}(x_j).
\]
For wavenumber vectors $\bsk = (k_1, \ldots, k_s) \in \nat_0^s$ and $\bsx =
(x_1,\ldots,x_s) \in [0,1)^s$ we write
\[
    _p\wal_{\bsk}(\bsx) :=\, _p\wal_{k_1,\ldots,k_s}(x_1,\ldots,x_s).
\]
Again, if it is clear which base is meant we simply write
$\wal_{\bsk}(\bsx)$.
\end{definition}

Let a Walsh series $f \in \LL_2([0,1]^s)$ be defined by  $$f(\bsx) =
\sum_{\bsk \in \nat_0^s} \widehat{f}(\bsk) \wal_{\bsk}(\bsx),$$
where the Walsh coefficients are given by
\begin{equation}\label{eq_walshcoeff}
\widehat{f}(\bsk) = \int_{[0,1]^s} f(\bsx)
\overline{\wal_{\bsk}(\bsx)} \rd \bsx,
\end{equation}
since the Walsh functions are mutually orthonormal.  Here $\overline{a}$ denotes the complex conjugate of a complex
number $a$. 

The aim now is to approximate Walsh coefficients of a function $f$ with wavenumbers lying in a
certain set of wavenumbers, $\KK(\CC)$, depending on the digital net $P(\CC)$ defined by the
generating matrices $\CC = (C_1,\ldots, C_s)$.  The details
of how $\KK(\CC)$ is chosen is explained in the next section. A
digital net with $p^m$ points can be used to estimate $|\KK(\CC)| = p^m$ Walsh coefficients, where $|\cdot|$ denotes the cardinality of a set. For $\bsk
\in \KK(\CC)$ we approximate $\widehat{f}(\bsk)$ by the finite sum
\begin{equation}\label{eq_dwalshcoeff}
\widetilde{f}(\bsk) = \frac{1}{p^m} \sum_{\bsx \in P(\CC)} f(\bsx)
\overline{\wal_{\bsk}(\bsx)}.
\end{equation}
We call $\widetilde{f}(\bsk)$ the discrete Walsh coefficients
because \eqref{eq_dwalshcoeff} is just a discrete version of
\eqref{eq_walshcoeff}. Those discrete Walsh coefficients provide us
with valuable information about the function at hand.

A naive calculation of the $p^m$ discrete Walsh coefficients
$\widetilde{f}(\bsk)$ with $\bsk \in \KK(\CC)$ would require
$\landau(p^{2m})$ operations, but using the fast discrete Walsh
transform algorithm described in the following section, we can reduce it to $\landau(m p^{m} \log p)$
operations. In Section~\ref{sec_int} the discrete
Walsh coefficients are used to interpolate functions based on observations on the digital net design.  The inverse discrete Walsh
transform then provides an interpolatory approximation of the original function. The ANOVA decomposition of this interpolation provides information about the effective dimension of the
function as explained in Section~\ref{sec_anova}. In the last
section the fast discrete Walsh transform is used to
approximate Walsh coefficients and effective dimensions of some explicit test functions.

\section{Multivariate fast discrete Walsh transform over digital nets}

In this section we introduce an FFT-like algorithm for a
multivariate fast discrete Walsh transform over digital nets. An
essential role is played by the dual net of a digital net.

\subsection{The Dual Net}

Let $\CC = (C_1,\ldots,C_s)$ be the vector of  matrices generating 
a digital $(t,m,s)$-net $P(\CC)$ over a finite field $\FF_p$, where
$p$ is prime. For any wavenumber vector $\bsk= (k_1,\ldots,k_s) \in \nat_0^s$, we define
$$\CC \cdot \bsk = C_1^T \vec{k}^{(m)}_1 + \cdots +
C_s^T\vec{k}^{(m)}_s \in
\FF_p^m,$$ where for $k_j = k_{j,0} + k_{j,1} p +
\cdots$ we define $\vec{k}^{(m)}_j = (k_{j,0},\ldots,k_{j,m-1})^T \in
\FF_p^m$ as the $m$-element truncation of $\vec{k}_j$,
and all operations are carried out in the finite field $\FF_p$. The
dual net $\DD(\CC)$ of the
digital net $P(\CC)$ is the set of all wavenumbers which make this dot product zero:
$$\DD(\CC) = \{\bsk \in \nat_0^s :
\CC \cdot \bsk = \vec{0} \}.$$ 
The dual net appears in the
worst-case error for multivariate integration in certain Walsh
spaces, see \cite{DP03a}. Therein the worst-case error is just the
sum of a certain function 
over all elements in the dual net except $\bszero$.

The dual net satisfies
\begin{equation}\label{char} \frac{1}{p^m} \sum_{\bsx \in P(\CC)}
\wal_{\bsk}(\bsx) = \left\{\begin{array}{ll} 1 & \mbox{if } \bsk \in
\DD(\CC) \\
0 & \mbox{otherwise},
\end{array}\right. \end{equation}
which was shown, for example, in \cite{DP03a}. Indeed, this property
could also
be used to define the dual net.

For $k \in \nat_0$ with $k = k_0 + k_1 p + \cdots$ let $\nu(0) =
0$ and for $k > 0$ let $\nu(k) = 1 + \max\{i : k_i \neq 0\}$.
For $\bsk = (k_1,\ldots,k_s) \in \nat_0^s$ define $\nu(\bsk) =
\sum_{i=1}^s \nu(k_i)$ (see \cite{N86, NP01}). The function $\nu$
is a norm on the elements in the wavenumber space. It depends on
the most significant bit of the coordinates and can be related to
the $t$-value of a digital net, see \cite{NP01}.

For non-negative integers $k,l \in \nat_0$ with $k = k_0 + k_1 p +
\cdots$ and $l = l_0 + l_1 p + \cdots$ a digit-wise addition and
subtraction in base $p$ can be defined by $k \oplus l = a_0 + a_1 p
+ \cdots$ where $a_i \equiv k_i + l_i \pmod{p}$ and $k \ominus l =
b_0 + b_1 p + \cdots$ where $b_i \equiv k_i - l_i \pmod{p}$. For
non-negative integer vectors the digit-wise addition and subtraction
are defined component-wise.

Using this digit-wise addition and subtraction we obtain a group
structure on $\NN_0^s$ of which the dual net $\DD(\CC)$ forms a
subgroup. It can be checked that the cosets of the subgroup
$\DD(\CC)$ are given by $$\DD(\vec{h}) = \{\bsk \in \nat_0^s : \CC
\cdot \bsk = \vec{h}\}$$ for $\vec{h} \in \FF_p^m$ and hence there
are $p^m$ cosets. Note that $\DD(\vec{0}) = \DD(\CC)$, i.e., the coset containing $\vec{0}$ is the dual net. The set of wavenumbers whose Walsh coefficients are to be approximated,
$\KK(\CC)$, is then obtained by choosing exactly one representative
in each coset. For each $\vec{h} \in \FF_p^m$ identify $\bsk \in
\DD(\vec{h})$ such that $\nu(\bsk) \le \nu(\bsl)$ for all $\bsl \in
\DD(\vec{h})$.  This $\bsk$ is the representative of $\DD(\vec{h})$ chosen to be in $\KK(\CC)$.  In the case of more than one $\bsk$ from the same coset satisfying this condition, one may choose, for example, the $\bsk$ that is the smallest in lexicographic order.  That is,  
\begin{multline*}
\KK(\CC) = \left \{ \bsk \in \nat_0^s : \bsk \in \DD(\vec{h}) \text{ for some } \vec{h} \in \FF_p^m, \text{ and for any } \bsl \ \in
\DD(\vec{h}) \right. \\ \left. \text{ we have } \nu(\bsk) \le \nu(\bsl) 
\text{ and if } \nu(\bsk) = \nu(\bsl) \text{ for some } \bsl \in \DD(\vec{h}), \text{ then } \right. \\ \left. k_1 = l_1, \ldots, k_{j-1} = l_{j-1}, k_j < l_j \text{ for some } j=1, \ldots, s. \right\}
\end{multline*}
This definition implies that the zero vector
$\bszero$ is automatically in $\KK(\CC)$, and that $\nat_0^s$ is the direct sum of $\KK(\CC)$ and $\DD(\CC)$.

\subsection{Multivariate fast discrete Walsh transform over digital nets}

Let a Walsh series $f \in \LL_2([0,1]^s)$ be given by  $$f(\bsx) =
\sum_{\bsk \in \nat_0^s} \widehat{f}(\bsk) \wal_{\bsk}(\bsx).$$ For all $\bsk \in \KK(\CC)$ one may approximate $\widehat{f}(\bsk)$ by the discrete Walsh transform (DWT), $\widetilde{f}(\bsk)$, which is defined as 
\begin{equation*}
\widetilde{f}(\bsk) = \frac{1}{p^m} \sum_{\bsx \in P(\CC)} f(\bsx)
\overline{\wal_{\bsk}(\bsx)} = \sum_{\bsl \in \nat_0^s}
\widehat{f}(\bsl) \frac{1}{p^m} \sum_{\bsx \in P(\CC)}
\wal_{\bsl}(\bsx) \overline{\wal_{\bsk}(\bsx)}.
\end{equation*}
Note that $\wal_{\bsl}(\bsx) \overline{\wal_{\bsk}(\bsx)} =
\wal_{\bsl\ominus \bsk}(\bsx)$ (see \cite{DP03a}), and hence the
rightmost sum in the equation above is one if $\bsl \ominus \bsk \in
\DD(\CC)$ and zero otherwise. Thus it follows that
\begin{equation}\label{eq_aliasing}
\widetilde{f}(\bsk) = \widehat{f}(\bsk) + \sum_{\bsl \in
\DD(\CC)\setminus \{\bszero\}} \widehat{f}(\bsk \oplus \bsl).
\end{equation}
Hence, the terms $\widehat{f}(\bsk \oplus \bsl)$ are completely aliased
with each other for all $\bsl \in \DD(\CC)$. We have chosen $\bsk
\in \KK(\CC)$ such that $\bsk$ is closest to $\bszero$.  Hence if higher frequency contributions are sufficiently
small, that is, $\widehat{f}(\bsk)$ decays sufficiently fast the
further $\bsk$ is away from $\bszero$ with respect to the norm
$\nu$, then $\widetilde{f}(\bsk) \approx \widehat{f}(\bsk)$.

A straightforward calculation of the discrete Walsh coefficients would require $\landau(p^{2m})$ operations, as we have $p^m$ sums
to compute (one sum for each $\bsk \in \KK(\CC)$) and each sum
requires $\landau(p^m)$ operations. But as shown below, certain
parts in the summation above can be reused and thereby reducing
the number of operations.

Let $\bsx_n$ be the $n$-th point of the digital net $P(\CC)$ and
let $\bsk$ be the unique element of $\DD(\vec{h}) \cap \KK(\CC)$. Then we have
$$\wal_{\bsk}(\bsx_n) = \omega_p^{\vec{n} \cdot \vec{h}},$$ because
$$\sum_{j=1}^s \sum_{i=1}^\infty y_{j,n,i} k_{j,i} = (\vec{k}_1^T
C_1 + \cdots + \vec{k}_s^T C_s) \vec{n} = \vec{h}^T \vec{n}.$$
Hence $$\widetilde{f}(\bsk) = \frac{1}{p^m} \sum_{n=0}^{p^m-1}
f(\bsx_n) \omega_p^{-\vec{n} \cdot \vec{h}},$$ where $\vec{h} = \CC
\cdot \bsk$. The above sum may be written as $$\widetilde{f}(\bsk) =
\frac{1}{p^m} \sum_{n_{m-1}=0}^{p-1} \omega_p^{-n_{m-1}h_{m-1}}
\cdots \sum_{n_1=0}^{p-1} \omega_p^{-n_1 h_1} \sum_{n_0=0}^{p-1}
f(\bsx_n) \omega_p^{-n_0 h_0}.$$ Now, computing first the
innermost sum for each $h_0 = 0,\ldots,p-1$, then the second
innermost sum for each $h_1 = 0,\ldots, p-1$ and so on yields an
algorithm which needs only $\landau(m p^{m+1})$ operations. The
details are given as follows.

\begin{algorithm}[Fast Discrete Walsh transform (FWT)]\label{alg_fwt}
For $n_0,\ldots, n_{m-1} \in \FF_p$ we define
$G^{(0)}(n_0,\ldots,n_{m-1}) = f(\bsx_{n_0+\cdots +
n_{m-1}p^{m-1}})$. Then for $r = 1,2,\ldots, m$ compute for all
$n_r,\ldots, n_{m-1} \in \FF_p$ and all $h_0,\ldots,h_{r-1} \in
\FF_p$ the sums
\begin{multline} \label{eqn:fdwt}
G^{(r)}(h_0,\ldots,h_{r-1},n_r,\ldots,n_{m-1})\\ =
\sum_{n_{r-1}=0}^{p-1} \omega_p^{-n_{r-1} h_{r-1}}
G^{(r-1)}(h_0,\ldots,
h_{r-2},n_{r-1},n_r,\ldots,n_{m-1}).
\end{multline}
For $\bsk \in \KK(\CC)$ with $\CC \cdot \bsk = \vec{h}$ let
$$\widetilde{f}(\bsk) = \frac{1}{p^m} G^{(m)}(h_0,\ldots, h_{m-1}).$$
\end{algorithm}

Note that in each step one needs $\landau(p^{m+1})$ operations, and
as there are $m$ steps, one needs $\landau(m p^{m+1})$ operations
altogether.  Note also that the number of terms in the summation in the right side of (\ref{eqn:fdwt}) is $p$, which is a prime number. The index $n_{r-1} = 1,\ldots,p-1$ (and $h_{r-1}$) forms a group under the multiplication modulo $p$. Thus, following the ideas of Rader's algorithm \cite{Rader}, we can rewrite the indices as $n_{r-1} = g^\alpha  \mod p$ and $h_{r-1} = g^\beta  \mod p$, where $g$ is a primitive root of this group, and $\alpha = 0,\ldots,p-2$, $\beta = 0,\ldots,p-2$. By applying Rader's algorithm, we can further reduce the total number of steps to $\landau(m p^{m} \log p)$ \cite{Rader}.

\section{Function Interpolation}\label{sec_int}

In this section we consider multivariate spline interpolation over
digital nets using the discrete Walsh coefficients described in the previous section. Multivariate spline interpolation over lattice rules
was considered in \cite{ZengLeungHickernell2004}. See also
\cite{whaba} for more information on properties of splines.

\subsection{Reproducing kernel Walsh space}

Before we introduce the interpolation algorithm we introduce
reproducing kernel Hilbert spaces based on Walsh functions, see
\cite{DP03a}.  In the following we define the weighted Hilbert space
$\mathcal{H}_{K}$ based on Walsh functions.

Consider the set of functions
\[
    \mathcal{H}_{0,K} = \left\{f: f(\bsx) = \sum_{i=0}^{n'-1}
\alpha_i K(\bsx,\bsx'_i) : n' \in \nat_0, \ \alpha_i \in \mathbb{R}, \{\bsx'_i\}_{i=0}^{n'-1} \subset [0,1)^s\right\},
\]
defined in terms of a symmetric, positive definite kernel
function $K:[0,1)^{2s} \to \complex$. The kernel allows us to define an inner product on $\mathcal{H}_{0,K}$ as
\[
    \langle f,g \rangle = \sum_{i=0}^{n_f-1} \sum_{j=0}^{n_g -1}
\alpha_i \beta_j K(\bsx'_i,\bsy'_j)
\]
for any two functions
$f = \sum_{i=0}^{n_f-1} \alpha_i K(\cdot,\bsx'_i)$ and $g = \sum_{i=0}^{n_g-1}
\beta_i K(\cdot,\bsy'_i)$.  The linear space $\mathcal{H}_{0,K}$ may then be completed to obtain a Hilbert space, $\mathcal{H}_{K}$, for which $K$, is the reproducing kernel, see \cite{aron},
\[
K(\cdot, \bsy) \in \mathcal{H}_{K}, \quad f(\bsy) = \langle K(\cdot, \bsy),f \rangle \qquad \forall \bsy \in [0,1)^s, \forall f \in \mathcal{H}_{K}.
\]

The kernel functions considered here are based on Walsh
functions and may be written as
\begin{equation*}
K(\bsx,\bsy) = K(\bsx \ominus \bsy,\boldsymbol{0}) = \sum_{\bsk\in \nat_0^s}
\widehat{K}(\bsk) \mathrm{wal}_{\bsk}(\bsx \ominus
\bsy),
\end{equation*}
where positive definiteness is ensured by requiring $\widehat{K}(\bsk)$ to be real and non-negative for all $\bsk \in \nat_0^s$.  The finiteness of this kernel is ensured by requiring that $\widehat{K}$ be summable.
The inner product for the Hilbert space defined by this kernel may be written as an $\ell_2$ inner product in the spectral domain:
\begin{equation*}
    \langle f,g\rangle_{\mathcal{H}_K} = \sum_{\bsk\in \nat_0^s}
            \frac{\widehat{f}(\bsk) 
\overline{\widehat{g}(\bsk)} }
{\widehat{K}(\bsk)} = \left \langle  \frac{\widehat{f}}
{\sqrt{\widehat{K}}},  \frac{\widehat{g} }
{\sqrt{\widehat{K}}} \right \rangle_2
\end{equation*}
where $\widehat{f}$ and $\widehat{g}$ are the
Walsh coefficients of $f$
and $g$. The accompanying norm is $\|f\|_{\mathcal{H}_K} =
\langle f,f\rangle_{\mathcal{H}_K}^{1/2}$.

\subsection{Interpolation of functions in the Walsh space}

We now interpolate functions in $\mathcal{H}_K$ using a linear
combination of the reproducing kernel function where the second variable is fixed, $K(\cdot, \bsx_n)$, $n = 0,\ldots, p^m-1$, and where the $\bsx_n$ are points taken from a digital net, $\mathcal{P}(\mathcal{C})$. Given the values of
a function $f(\bsx_n), n=0,\ldots,p^m-1$, one can approximately recover it by a spline, defined as
\begin{equation*}
Sf(\bsx) = \sum_{n=0}^{p^m-1} c_n K(\bsx,\bsx_n),
\end{equation*}
where the $c_n$ are the coefficients to be found by interpolation: $f(\bsx_n) = Sf(\bsx_n)$ for $n
= 0,\ldots, p^m-1$. This translates into solving the linear system
\begin{equation*}
f(\bsx_n) = \sum_{v=0}^{p^m-1} c_v K(\bsx_n,\bsx_v) \quad \mbox{for
} n = 0,\ldots, p^m-1,
\end{equation*}
for the coefficients $c_0,\ldots, c_{p^m-1}$ given the  $f(\bsx_n)$ and $K(\bsx_n,\bsx_v)$. In the following paragraphs we show that the
$c_n$ can be computed in  $\Order(m p^{m} \log p)$
operations. See \cite{zin87} for an analogue for lattice rules in
the context of Fredholm integral equations of the second kind.

First observe that for any $\bsk \in \mathcal{K}(\mathcal{C})$ the DWT of the function data is
\begin{align}\label{eq_ctilde}
\widetilde{f}(\bsk) & =  \frac{1}{p^m} \sum_{n=0}^{p^m -1}
f(\bsx_n) \overline{\mathrm{wal}_{\bsk}(\bsx_n)} =  \frac{1}{p^m} \sum_{n=0}^{p^m -1} Sf(\bsx_n)
\overline{\mathrm{wal}_{\bsk}(\bsx_n)} \nonumber \\
 & =  \frac{1}{p^m} \sum_{\bsl\in \nat_0^s}
\widehat{K}(\bsl) \sum_{r=0}^{p^m-1} c_r
\overline{\mathrm{wal}_{\bsk}(\bsx_r)} \sum_{n=0}^{p^m
-1} \mathrm{wal}_{\bsl}(\bsx_n)
\overline{\mathrm{wal}_{\bsk}(\bsx_n)} \nonumber \\
 & =  p^m \sum_{\bsd\in
\DD(\CC)}\widehat{K}(\bsk \oplus \bsd)
\widetilde{c}(\bsk) =  p^m \widetilde{K}(\bsk) \widetilde{c}(\bsk),
\end{align}
where $\widetilde{c}(\bsk) = p^{-m} \sum_{n=0}^{p^m-1} c_n
\overline{\wal_{\bsk}(\bsx_n)}$ is the DWT of the coefficients, $(c_n)_{n=0}^{p^m-1}$, and $\widetilde{K}(\bsk) =
\sum_{\bsl \in \DD(\CC)} \widehat{K}(\bsk \oplus \bsl)$ is the DWT of the kernel data, $(K(\bsx_n,\bszero))_{n=0}^{p^m-1}$.

This implies that the DWT of the coefficients is essentially the quotient of the DWTs of the function data and the kernel data:
\begin{equation} \label{splinecoef}
\widetilde{c}(\bsk) =
\frac{\widetilde{f}(\bsk)}{p^m \widetilde{K}(\bsk)}.
\end{equation}
Using the FWT algorithm we can compute all
$(\widetilde{f}(\bsk))_{\bsk \in \KK(\CC)}$ and $(\widetilde{K}(\bsk))_{\bsk \in \KK(\CC)}$, and hence $(\widetilde{c}(\bsk))_{\bsk \in \KK(\CC)}$ in $\Order(mp^{m} \log p)$ operations.  To compute $c_0,\ldots, c_{p^m-1}$ requires the inversion of the DWT, which is introduced in the next section.

\subsection{Fast inverse discrete Walsh transform}

The following lemma gives the key to the inverse discrete Walsh
transform over a digital net.

\begin{lemma} \label{idwtlem}
Let $P(\CC)=\{\bsx_0, \ldots, \bsx_{p^m-1}\}$ be a digital net with $p^m$ points. Let $c_0,\ldots,
c_{p^m-1}$ be arbitrary complex numbers and let
$(\widetilde{c}(\bsk))_{\bsk \in \KK(\CC)}$, denote the DWT of
$(c_n)_{n=0}^{p^m-1}$. Then for $n = 0,\ldots, p^m-1$ the coefficients are
\begin{equation*}
c_n = \sum_{\bsk \in \KK(\CC)} \widetilde{c}(\bsk) \mathrm{wal}_{\bsk}(\bsx_n).
\end{equation*}
\end{lemma}

\begin{proof}
The sum on the right of the above equation is 
\begin{align*}
\sum_{\bsk \in \KK(\CC)} \widetilde{c}(\bsk) \wal_{\bsk}(\bsx_n) & =
\sum_{v = 0}^{p^m-1} c_v \frac{1}{p^m} \sum_{\bsk \in \KK(\CC)}
\overline{\wal_{\bsk}(\bsx_v)} \wal_{\bsk}(\bsx_n) \\ 
& = \sum_{v
= 0}^{p^m-1} c_v \frac{1}{p^m} \sum_{\bsk \in \KK(\CC)}
\wal_{\bsk}(\bsx_n \ominus \bsx_v).
\end{align*}
The definition of the Walsh function and the net imply that  $\wal_{\bsk}(\bsx_n \ominus \bsx_v) =
\wal_{\bsk}(\bsx_{n \ominus v}) = \omega_p^{(\vec{n} \ominus
\vec{v}) \cdot \vec{h}}$, where $\bsk \in \DD(\vec{h}) \cap
\KK(\CC)$. Thus we have
$$\frac{1}{p^m} \sum_{\bsk \in \KK(\CC)} \wal_{\bsk}(\bsx_n \ominus
\bsx_v) = \frac{1}{p^m} \sum_{\vec{h} \in \FF_p^m}
\omega_p^{(\vec{n} \ominus \vec{v}) \cdot \vec{h}}$$ and the last
sum is $1$ if $\vec{n} = \vec{v}$, i.e. $n = v$, and $0$ otherwise.
Hence the result follows.
\end{proof}

The above lemma describes the inversion of the discrete Walsh
transform which we might call the inverse discrete Walsh transform.
As for the discrete Walsh transform, there is also a fast inversion
of the discrete Walsh transform which we describe in the following.

\begin{algorithm}[Fast Inverse Discrete Walsh Transform (FIWT)]\label{alg_ifwt}
For $(h_0,\ldots, h_{m-1} \in \FF_p$ we define
$D^{(0)}(h_0,\ldots,h_{m-1}) = \widetilde{c}(\bsk)$ where $\bsk \in
\DD(\vec{h}) \cap \KK(\CC)$. Then for $r =
1,2,\ldots, m$ compute for all $h_r,\ldots, h_{m-1} \in \FF_p$ and
all $n_0,\ldots,n_{r-1} \in \FF_p$ the sums
\begin{multline*}
D^{(r)}(n_0,\ldots,n_{r-1},h_r,\ldots,h_{m-1}) \\ =
\sum_{h_{r-1}=0}^{p-1} \omega_p^{h_{r-1} n_{r-1}}
D^{(r-1)}(n_0,\ldots,
n_{r-2},h_{r-1},h_r,\ldots,h_{m-1}).
\end{multline*}
Then for $n = 0,\ldots, p^m-1$ with $n = n_0 + \cdots + n_{m-1}
p^{m-1}$ let
$$c_n = D^{(m)}(n_0,\ldots, n_{m-1}).$$
\end{algorithm}

Hence $c_0,\ldots, c_{p^m-1}$ can also be computed from
$(\widetilde{c}(\bsk))_{\bsk \in \KK(\CC)}$ in $\Order(m p^{m+1})$
operations. As with the FWT algorithm, this number of operations can be further reduced to $\landau(m p^{m} \log p)$ if we apply Rader's algorithm. Thus the spline interpolation at a point, $Sf(\bsx)$, can be computed in
$\landau(m p^{m} \log p)$ operations, and each additional point can be computed in $\landau(p^{m})$ operations.

\subsection{Best possible interpolation}

Splines as defined above provide optimal interpolation of a function in at least two senses.  The spline approximation is the smallest function in $\HH_K$ that interpolates the data:
\[
Sf = \min_{\substack{g \in \HH_K \\ g(\bsx_i)=f(\bsx_i),\ i=0, \ldots, p^m}} \norm[\HH_K]{g}.
\]
Moreover, the spline algorithm is the best linear algorithm, i.e., algorithm of the form $Af = \sum_{i=0}^{n-1} f(\bsx_i) w_i(\bsx)$ for any choice of the $w_i(\bsx)$:
\[
Sf(\bsx) = \argmin_{Af} \sup_{\norm[\HH(K)]{f} \le 1} \left\lvert f(\bsx) -  Af(\bsx) \right\rvert.
\]
The proofs of these assertions are contained in \cite[Chapter 18]{Fas07a} and elsewhere. 

\section{ANOVA Decomposition of the Interpolation}\label{sec_anova}

Based on the spline interpolation of a function, one can also approximate its analysis of variance
 (ANOVA) effects, i.e., the pieces of the function depending on a subset of the $s$ variables.  This provides a way to estimate the truncation and superposition dimensions of functions via their spline interpolations.

Let $1:s$ denote the set of coordinate indices, $\{1,\ldots,s\}$, for short.  Let $u$ denote a subset of $1:s$ and let $\bar{u}$ denote $1:s \setminus u$. For any $\bsx \in [0,1]^s$, let $\bsx_u = (x_j)_{j \in u}$ denote the vector of coordinates indexed by $u$.  The ANOVA decomposition \cite{ES,Sob} of a function $f:[0,1]^s \longrightarrow \mathbb{R}$, is denoted
\begin{subequations} \label{ANOVAdef}
\begin{equation}
f(\bsx) = \sum_{u \subseteq 1:s} f_u(\bsx_{u}),
\end{equation}
where the ANOVA effect, $f_u$, is defined recursively by taking the integral over $[0,1]^{\bar{u}}$ and then subtracting the lower order effects:
\begin{equation}
f_{\emptyset} = \int_{[0,1]^{s}} f(\bsx) \,
\rd \bsx, \qquad  f_{u}(\bsx_u) = \int_{[0,1]^{\bar{u}}}f(\bsx)
\, \rd \bsx_{\bar{u}} - \sum_{v\subset u} f_{v}(\bsx_v),
\end{equation}
\end{subequations}
We emphasize that $\subset$ on the right side of this last equation denotes the proper subset. Also, $[0,1]^{u}$ denotes the Cartesian product of $|u|$ copies of $[0,1]$, where $|u|$ is the cardinality of $u$.

The ANOVA effects, $f_u$, of the function $f \in \HH_K$ lie in subspaces, $\HH_{K_u}$ for kernels constructed appropriately.  The kernels $K_u$ are products of a univariate kernel, $K'$.

The kernel $K'$ for univariate functions is defined as
\begin{equation*}
K'(x,y) = K'(x\ominus y,0) = \sum_{k=1}^\infty \widehat{K'}(k)\mathrm{wal}_k (x \ominus
y),
\end{equation*}
where the Walsh coefficients of the kernel must be non-negative.  One reasonable choice is 
\[
\widehat{K'}(k) =  \frac{p^{\alpha} - p}{p^{\alpha}(p-1)} p^{-\alpha a}
\]
for $k = k_0 + k_1 p + \cdots + k_a p^a \in \nat$ with $k_a \neq 0$, and where $\alpha > 1$ is a parameter that measures the digital smoothness of the kernel.  These Walsh coefficients have been normalized so that 
\begin{equation*}
K'(x,x) = K'(0,0) = \sum_{k=1}^\infty \widehat{K'}(k) = 1.
\end{equation*}
Because $K'$ does not include the constant function, i.e., $\widehat{K'}(0)=0$ implicitly, it follows that $\int_0^1 K'(x,y) \, \rd y = 0$.
A computable short form of $K'$ can be obtained \cite{DP03a},
namely for $x = x_1 p^{-1} + x_2 p^{-2} + \cdots $ and $y = x_1
p^{-1} + x_2 p^{-2} + \cdots + x_{i-1} p^{-i+1} + y_{i} p^{-i} + y_{i+1} p^{-i-1} + \cdots $ with $y_i \ne x_i$
we have
\begin{equation*}
K'(x,y) = K'(x\ominus y,0) = 
1 - p^{i(1 - \alpha)}\frac{p^{\alpha}-1}{p-1}.
\end{equation*}

The kernel for functions of $s$ variables is a product involving $K'$, namely,
\begin{align*}
K(\bsx,\bsy) &= K(\bsx \ominus \bsy, \bszero ) = \sum_{\bsk\in \nat_0^s}
\widehat{K}(\bsk) \mathrm{wal}_{\bsk}(\bsx \ominus
\bsy) \\
&= \prod_{j=1}^s [1 + \gamma_j K'(x_j,y_j)]  = \prod_{j=1}^s [1 + \gamma_j K'(x_j \ominus y_j,0)] \\
&= \sum_{u\subseteq 1:s} \gamma_u K_u(\bsx_u , \bsy_u) = \sum_{u\subseteq 1:s} \gamma_u K_u(\bsx_u \ominus \bsy_u,\bszero),
\intertext{where the tuning parameter $\gamma_u$ has a product form,}
\gamma_u &= \prod_{j \in u} \gamma_j, \qquad \gamma_{\emptyset}=1, 
\intertext{and the Walsh coefficients of these multivariate kernels have expressions in terms of the Walsh coefficients of $K'$:}
\widehat{K}(\bsk) & = \prod_{j=1}^s [\delta_{k_j,0} + \gamma_j \widehat{K'}(k_j)] = \sum_{u \subseteq 1:s} \gamma_u \widehat{K_u}(\bsk_u) \delta_{\bsk_{\bar{u},\bszero}}, \\
K_u(\bsx_u , \bsy_u) &=  K_u(\bsx_u \ominus \bsy_u,\bszero)  = \sum_{\bsk_u \in \nat_0^u}
\widehat{K_u}(\bsk_u) \mathrm{wal}_{\bsk_u}(\bsx \ominus
\bsy) \\
&= \prod_{j \in u} K'(x_j,y_j) = \prod_{j \in u} K'(x_j \ominus y_j,0), \\
\widehat{K_u}(\bsk_u) & = \prod_{j \in u} \widehat{K'}(k_j), \qquad \widehat{K_{\emptyset}}=1.
\end{align*}

The spline approximation via this kernel may be expressed as the sum of its ANOVA effects:
\begin{equation*}
Sf(\bsx) = \sum_{n=0}^{p^m-1} c_n
K(\bsx , \bsx_{n})= \sum_{n=0}^{p^m-1} c_n
K(\bsx \ominus \bsx_{n},\bszero) = \sum_{u\subseteq 1:s} (Sf)_u(\bsx_u), 
\end{equation*}
where these ANOVA effects are written in terms of the kernels $K_u$:
\begin{equation}\label{eq_Sfu}
(Sf)_u(\bsx_u) =  \gamma_u \sum_{n=0}^{p^m-1} c_n K_u (\bsx_u
, \bsx_{n,u}) = \gamma_u \sum_{n=0}^{p^m-1} c_n K_u (\bsx_u
\ominus \bsx_{n,u}, \bszero).
\end{equation}
These $(Sf)_u$ may be verified as the ANOVA effects of $Sf$ defined in \eqref{ANOVAdef} by noting that $\int_0^1 K'(x,y) \, \rd y = 0$ for all $x$.

The special form of the reproducing kernel defined here facilitates the calculation of the variance of $(Sf)_u$, denoted $\sigma^2((Sf)_u)$.  Noting that for $|u| >0$, $(Sf)_u$ has zero mean, it follows that
\begin{align*}
\sigma^2((Sf)_u) & =  \int_{[0,1]^{u}} \left[Sf_u(\bsx_u) -
\int_{[0,1]^{u}} Sf_u(\bsx'_u) \rd \bsx'_u \right]^2 \rd \bsx_u \\
& = \int_{[0,1]^{u}} [Sf_u(\bsx_u)]^2 \rd \bsx_u\\
& = \gamma_u^{2} \sum_{n,v = 0}^{p^m-1} c_n
c_v \int_{[0,1]^{u}} K_u(\bsx_u,\bsx_{n,u})
K_u(\bsx_u,\bsx_{v,u}) \, \rd \bsx_u \\ 
& = \gamma_u^{2}
\sum_{n,v = 0}^{p^m-1} c_n c_v \prod_{j \in u} \int_0^1
K'(x_j,x_{n,j}) K'(x_j,x_{v,j}) \, \rd x_j.
\end{align*}

Substituting the Walsh expansions of the univariate kernels in the above integral and noting that the Walsh functions are orthogonal yields an expression for the integral in terms of a related univariate kernel:
\[
\int_0^1 K'(x_j,x_{n,j}) K'(x_j,x_{v,j}) \,\rd
x_j = R'(x_{n,j},x_{v,j}),
\]
where
\begin{align*}
R'(x,y) &= R'(x\ominus y,0) =\sum_{k=0}^\infty \widehat{R'}(k) \wal_k(x \ominus y), \qquad  \widehat{R'}(k) = \left [\widehat{K'}(k) \right]^2,\\
\widehat{R'}(0) &= 0 \quad \text{and} \quad \widehat{R'}(k) = \left[\frac{p^{\alpha} - p}{p^{\alpha}(p-1)}\right]^2 p^{-2\alpha a}
\end{align*}
for $k = k_0 + k_1 p + \cdots + k_a p^a \in \nat$ with $k_a \neq 0$.
Moreover, as is the case for $K'$, there is a computationally simple form for $R'$ as well.  For $x = x_1 p^{-1} + x_2 p^{-2} + \cdots $ and $y = x_1
p^{-1} + x_2 p^{-2} + \cdots + x_{i-1} p^{-i+1} + y_{i} p^{-i} + y_{i+1} p^{-i-1} + \cdots $ with $y_i \ne x_i$
we have
\begin{equation*}
R'(x,y) = R'(x\ominus y,0) = \frac{(p^{\alpha} - p)^2}{(p-1)(p^{2\alpha} - p)} \left[1 - p^{i(1 - 2\alpha)}\frac{p^{2\alpha}-1}{p-1} \right].
\end{equation*}
The product of the univariate kernels $R'$ is then used to define the kernels $R_u$:
\begin{align*}
R_u(\bsx_u , \bsy_u) &=  R_u(\bsx_u \ominus \bsy_u,\bszero)  = \prod_{j \in u} R'(x_j,y_j) = \sum_{\bsk\in \nat_0^u}
\widehat{R_u}(\bsk_u) \mathrm{wal}_{\bsk_u}(\bsx_u \ominus
\bsy_u), \\
\widehat{R_u}(\bsk_u) & = \prod_{j \in u} \widehat{R'}(k_j).
\end{align*}
This definition allows the variance of $(Sf)_u$ to be written in terms of the kernel $R_u$: 
\[
\sigma^2((Sf)_u) = \gamma_u^{2}
\sum_{n,v = 0}^{p^m-1} c_n c_v R_u(\bsx_{n,u}, \bsx_{v,u}).
\]

The sum above may be evaluated naively using $\landau(p^{2m})$ operations.  However, using the FWT allows evaluation with only $\landau(m p^{m} \log p)$ operations.  First, the spline coefficients, $c_n$ and $c_v$ are written in terms of their inverse discrete Walsh transform coefficients via Lemma \ref{idwtlem}:
\begin{align*}
\sigma^2((Sf)_u) &= \gamma_u^{2}
\sum_{n,v = 0}^{p^m-1} \sum_{\bsk \in \KK(\CC)} \widetilde{c}(\bsk)
\wal_{\bsk}(\bsx_n) \sum_{\bsl \in \KK(\CC)} \overline{\widetilde{c}(\bsl)}
\overline{\wal_{\bsl}(\bsx_v)} R_u(\bsx_{n,u}, \bsx_{v,u}) \\
&= \gamma_u^{2}
\sum_{\bsk, \bsl \in \KK(\CC)} \widetilde{c}(\bsk) \overline{\widetilde{c}(\bsl)} \sum_{n,v = 0}^{p^m-1} R_u(\bsx_{n,u} \ominus \bsx_{v,u},\bszero) \wal_{\bsk}(\bsx_n) \overline{\wal_{\bsl}(\bsx_v)}.
\end{align*}
The double sum of the kernel $R_u$ may be further simplified by applying certain elementary properties of Walsh functions for any $\bsk,\bsl \in \KK(\CC)$:
\begin{multline*}
\sum_{n,v = 0}^{p^m-1} R_u(\bsx_{n,u} \ominus \bsx_{v,u},\bszero) \wal_{\bsk}(\bsx_n) \overline{\wal_{\bsl}(\bsx_v)} \\
= \sum_{n,v = 0}^{p^m-1} R_u(\bsx_{n,u} \ominus \bsx_{v,u},\bszero) \wal_{\bsk}(\bsx_n \ominus \bsx_{v,u}) \wal_{\bsk \ominus \bsl}(\bsx_v) \\
= \sum_{\bsx \in P(\CC)} R_u(\bsx_{u},\bszero) \wal_{\bsk}(\bsx) \sum_{n,v = 0}^{p^m-1}\wal_{\bsk \ominus \bsl}(\bsx_v) = p^{2m} \delta_{\bsk, \bsl} \widetilde{R_u}(\bsk)
\end{multline*}
where the DWT of $R_u$ is
\[
\widetilde{R_u}(\bsk) = \frac{1}{p^m} \sum_{\bsx \in P(\CC)} R_u(\bsx_{u},\bszero) \wal_{\bsk}(\bsx).
\]
Substituting the formula for the double sum of the kernel $R_u$ in terms of its DWT yields 
\begin{equation} \label{sig2sfu}
\sigma^2((Sf)_u) = \gamma_u^{2}
p^{2m} \sum_{\bsk \in \KK(\CC)} \left \lvert \widetilde{c}(\bsk) \right \rvert^2 \widetilde{R_u}(\bsk).
\end{equation}
Since the DWT of the spline coefficients and $R_u$ may each be calculated in $\landau(m p^{m} \log p)$ operations by the FWT algorithm, it follows that each $\sigma^2((Sf)_u)$ may be calculated in $\landau(m p^{m} \log p)$ operations.

Calculating the variances of all the $Sf_u$ may be too burdensome, since there are $2^s$ ANOVA effects.  However, sums of the  $\sigma^2((Sf)_u)$ are useful for determining the effective dimension of $f$ \cite{CMO}, which is an
important factor in the performance of quasi-Monte Carlo methods.

The truncation variance of order $d$, denoted $\sigma^2_{\trc}(f;d)$ is defined as sum of the variances of all ANOVA effects involving the first $d$ or fewer variables:
\[
\sigma^2_{\trc}(f;d) = \sum_{u \subseteq 1:d} \sigma^{2}(f_u).
\]
The superposition variance of order $d$, denoted $\sigma^2_{\sup}(f;d)$ is defined as sum of the variances of all ANOVA effects involving $d$ or fewer variables:
\[
\sigma^2_{\sup}(f;d) = \sum_{0< |u| \leq d} \sigma^{2}(f_u).
\]
It follows from these two definitions that the truncation variance is no more than the superposition variance, and for all $d$ between $0$ and $s$, 
\begin{multline*}
0 = \sigma^2_{\trc}(f;0) = \sigma^2_{\sup}(f;0) \le \sigma^2_{\trc}(f;d) \le \sigma^2_{\sup}(f;d) \\
\le \sigma^2_{\trc}(f;s) = \sigma^2_{\sup}(f;s) = \sigma^2(f).
\end{multline*}
The truncation dimension, $d_{\trc}$, and the superposition dimension, $d_{\sup}$ are defined as the smallest dimensions for which the truncation and superposition variances, respectively, are $99\%$ of the total variance of the function, i.e., 
\begin{gather*}
\sigma^2_{\trc}(f;d_{\trc} - 1) < 0.99 \sigma^2(f) \le \sigma^2_{\trc}(f;d_{\trc}), \\
\sigma^2_{\trc}(f;d_{\sup} - 1) < 0.99 \sigma^2(f) \le \sigma^2_{\sup}(f;d_{\sup}).
\end{gather*}

The truncation and superposition dimensions of a function may be estimated by the truncation and superposition dimensions of their spline approximations.  To do this requires computationally efficient formulas for the truncation and superposition variances.  The truncation variance may be written as
\begin{align*}
\sigma^2_{\trc}(Sf;d) &= \sum_{u \subseteq 1:d} \sigma^{2}((Sf)_u)
= \sum_{u \subseteq 1:d} \gamma_u^{2}
p^{2m} \sum_{\bsk \in \KK(\CC)} \left \lvert \widetilde{c}(\bsk) \right \rvert^2 \widetilde{R_u}(\bsk) \\
& = p^{2m} \sum_{\bsk \in \KK(\CC)} \left \lvert \widetilde{c}(\bsk) \right \rvert^2 \widetilde{R_d}(\bsk),
\intertext{where $\widetilde{R_d}$ is the DWT of the kernel $R_d$ defined as}
R_d(\bsx,\bsy) &= R_d(\bsx \ominus \bsy, \bszero )  \\
&= \sum_{u\subseteq 1:d} \gamma_u^{2} R_u(\bsx_u , \bsy_u) = \sum_{u\subseteq 1:d} \gamma_u^{2} R_u(\bsx_u \ominus \bsy_u,\bszero) \\
&= \prod_{j=1}^d [1 + \gamma^2_j R'(x_j,y_j)]  = \prod_{j=1}^d [1 + \gamma^2_j R'(x_j \ominus y_j,0)].
\end{align*}

By convention $R(\bsx,\bsy)$ is defined as $R_s(\bsx,\bsy)$.  Note that the total variance of the spline approximation to $f$ is given by taking $d=s$ above:
\begin{equation}\label{varSf}
\sigma^2(Sf) = p^{2m} \sum_{\bsk \in \KK(\CC)} \left \lvert \widetilde{c}(\bsk) \right \rvert^2 \widetilde{R}(\bsk) = p^{2m} \sum_{\bsk \in \KK(\CC)} \lvert \widetilde{f}(\bsk) \rvert^2 \frac{\widetilde{R}(\bsk)}{\lvert\widetilde{K}(\bsk) \rvert^2},
\end{equation}
where \eqref{splinecoef} has been used.

For $\gamma_j$ of the form $\beta\tilde{\gamma}_{j}$ and thus $\gamma_u = \beta^{|u|} \tilde{\gamma}_u$ the DWT $\widetilde{R}(\bsk)$ may be written as an $s$-degree polynomial in $\beta^2$ with vanishing constant term:
\begin{align*}
R(\bsx,\bsy) &= R(\bsx \ominus \bsy, \bszero ) = \sum_{j=1}^s \beta^{2j}  \sum_{\substack{u\subseteq 1:s \\ \lvert u \rvert=j}} \tilde{\gamma}_u^{2} R_u(\bsx_u \ominus \bsy_u,\bszero), \\
\widetilde{R}(\bsk) &= \sum_{j=1}^s \beta^{2j} \widetilde{Q}(\bsk,j) , 
\intertext{where the $\widetilde{Q}(\bsk,j)$ are the coefficients of this polynomial in $\beta^2$ and are given by}
\widetilde{Q}(\bsk,j) &= \tilde{\gamma}_u^{2} \sum_{\substack{u\subseteq 1:s \\ \lvert u \rvert=j}} \widetilde{R_u}(\bsk).
\end{align*}
The coefficients $\widetilde{Q}(\bsk,j)$ may be obtained by evaluating $\widetilde{R}(\bsk)$ at $s$ different values of $\beta^2$, and then performing polynomial interpolation, a relatively inexpensive procedure requiring $\landau(s^2)$ operations.  These $\widetilde{Q}(\bsk,j)$ may then be used to evaluate the superposition variance as follows:
\begin{align*}
\sigma^2_{\sup}(f;d) &= \sum_{0< |u| \leq d} \sigma^{2}((Sf)_u)
= \sum_{0< |u| \leq d} \gamma_u^{2}
p^{2m} \sum_{\bsk \in \KK(\CC)} \left \lvert \widetilde{c}(\bsk) \right \rvert^2 \widetilde{R_u}(\bsk) \\
& = p^{2m} \sum_{\bsk \in \KK(\CC)} \left \lvert \widetilde{c}(\bsk) \right \rvert^2 \sum_{0< |u| \leq d} \beta^{2|u|} \tilde{\gamma}_u^2 \widetilde{R_u}(\bsk) \\
& = p^{2m} \sum_{\bsk \in \KK(\CC)} \left \lvert \widetilde{c}(\bsk) \right \rvert^2 \sum_{j=1}^d  \beta^{2j}\widetilde{Q}(\bsk,j) \\
& = \sigma^2_{\sup}(f;d-1) + p^{2m} \beta^{2d} \sum_{\bsk \in \KK(\CC)} \left \lvert \widetilde{c}(\bsk) \right \rvert^2 \widetilde{Q}(\bsk,d).
\end{align*}
Thus, both the truncation and superposition dimensions may be evaluated via the FWT algorithm using $\landau(s^2 m p^m \log(p))$ operations.

\section{Numerical Results}

\subsection{Optimization of the Kernel Parameters}
The $\gamma_j$ are weights in the ANOVA decomposition of the kernel and $\alpha$ controls the rate of decay of Fourier coefficients of the kernel.  Their optimal values for spline interpolation depend on the particular function $f$ to be interpolated. Assume that $p=2$, and assume that we have a digital net with $2N$ points, $\{\xx_0, \ldots, \xx_{2N-1}\}$, whose first $N$ points themselves constitute a digital net, as do the second $N$ points. We can construct the spline interpolant by the values of $f(\xx_n)$ for $n=0, \ldots, N-1$, and then estimate the error of the spline by the cost function
 \begin{equation} \label{splineCostFunc}
     \sum_{n=N+1}^{2N-1}(f(\xx_n) - Sf(\xx_n))^2.
 \end{equation} 
The values of $Sf(\xx_n)$ for $n=N,\ldots,2N-1$ can be evaluated by a fast discrete Walsh transform.
 There are $s$ optimization parameters $\gamma_j$ and the optimization process becomes slow when $s$ increases. In order to
 reduce the number of optimization parameters, the values of 
 $\gamma_j$ are assumed to be of the following form in this section,
 \begin{equation*}
 \gamma_j = \beta j^{q}
 \end{equation*}
 where $q$ is a new optimization parameter. Thus, for any
 $s$, there are only three optimization parameters: $\beta$, $\alpha$
 and $q$.
 Their optimal values in Sections \ref{sec:effDimMultiFunc} and \ref{sec:effDimAsianOption} were found by optimizing (\ref{splineCostFunc}) using the function \texttt{fminsearch} in MATLAB.

\subsection{Effective Dimension of Multiplicative Functions}
\label{sec:effDimMultiFunc}
Consider a class of test functions in \cite{WangFang2003}
\begin{equation*}
f(\bsx) = \prod_{k=1}^{s} \frac{|4x_k -2| + a_k}{1 + a_k},
\end{equation*}
where $a_k$ are parameters. We consider three possible choices
of $a_k$: $a_k = 1$, $a_k = k$ and $a_k = k^2$ for $k = 1\ldots s$. 
The effective dimension in both senses of these
functions can be computed analytically, or estimated by the algorithm in \cite{WangFang2003}. We compare
them with the effective dimensions of the interpolating spline as described in the previous section.
Table \ref{tab:MultiFuncEffDim} shows the results using a Sobol
point set with $m=12$ and $p = 2$, so $N=4096$.

\begin{table}
\begin{center}
\begin{tabular}{c|c|ccc|ccc}
$a_k$ & $s$ & \multicolumn{3}{c|}{$d_{\trc}$} & \multicolumn{3}{c}{$d_{\sup}$} \\
 & & Exact & Wang \& Fang & Spline & Exact & Wang \& Fang &
Spline
 \\ \hline
           & 10 & 10 & 10 & 10 & 3 & - & 2 \\
$1$  & 20 & 20 & 20 & 20 & 5 & - & 2 \\
           & 40 & 40 & 39 & 40 & 8 & - & 2 \\ \hline
           & 10 & 10 & 10 & 10 & 2 & -- & 2 \\
$k$  & 20 & 18 & 19 & 18 & 2 & -- & 2 \\
           & 40 & 33 & 34 & 31 & 2 & -- & 2 \\ \hline
           & 10 & 5 & 5 & 5 & 2 & - & 2 \\
$k^2$& 20 & 5 & 5 & 5 & 2 & - & 2 \\
           & 40 & 5 & 5 & 5 & 2 & - & 2 \\
\end{tabular}
\end{center}
\caption{The effective dimensions of the multiplicative functions}
\label{tab:MultiFuncEffDim}
\end{table}

Note that $d_{\sup}$ cannot be calculated by the method in
\cite{WangFang2003}. The spline method described in the previous section can estimate both $d_{\sup}$ and
$d_{\trc}$ well in all cases except for $d_{\sup}$ in the case $a_k = 1$ and $s=20,40$. Table~\ref{tab:MultiFuncVar} shows $\sigma^2(f)$ by
analytical calculation and two approximations: $\sigma^2_{\mathrm{QMC}}(f) =
\frac{1}{N}\sum_{n=0}^{N-1}f^2(\bsx_n) -
(\frac{1}{N}\sum_{n=0}^{N-1}f(\bsx_n))^2$ and  $\sigma^2(Sf)$ by \eqref{varSf}.  Table~\ref{tab:MultiFuncVar} also shows the ratio $\sigma^2(Sf) / \sigma^2_{\mathrm{QMC}}(f)$ as an indicator of how well these two approximations agree. One can prove that $\sigma^2(Sf) / \sigma^2_{\mathrm{QMC}}(f) \leq 1$. It is found that the values of
$\sigma^2(Sf) / \sigma^2_{\mathrm{QMC}}(f)$ are close to unity for most cases but  are extremely small for  $a_k = 1,
s=20,40$. It was reported in \cite{WangHickernell2000} that the errors of the numerical integration by the QMC methods were also large for these latter cases. 

\begin{table}
\begin{center}
\begin{tabular}{c|c|cccc}
$a_k$ & $s$ & $\sigma^2(f)$ & $\sigma^2_{\mathrm{QMC}}(f)$ & $\sigma^2(Sf)$
&
$\sigma^2(Sf) / \sigma^2_{\mathrm{QMC}}(f)$ \\ \hline
           & 10 & 1.2265 & 1.9990 & 1.0434 & 0.5220   \\
$1$  & 20 & 3.9573 & $2.7 \times 10^3$ & 0.2692 & $9.9704 \times
10^{-5}$ \\
           & 40 & 23.5745 & $2.98 \times 10^{10}$ & 0.2393 & $8.0302
\times 10^{-12}$ \\ \hline
           & 10 & 0.1992 & 0.2025 & 0.1960 & 0.9679 \\
$k$  & 20 & 0.2154 & 0.2340 & 0.2073 & 0.8859 \\
           & 40 & 0.2246 & 0.3134 & 0.2088 & 0.6662 \\ \hline
           & 10 & 0.1038 & 0.1039 & 0.1037 & 0.9981 \\
$k^2$& 20 & 0.1039 & 0.1040 & 0.1038 & 0.9981 \\
           & 40 & 0.1039 & 0.1041 & 0.1038 & 0.9971 \\
\end{tabular}
\end{center}
\caption{The variances of the multiplicative functions estimated by a
QMC method and the variances of the splines.}
\label{tab:MultiFuncVar}
\end{table}

\subsection{Asian Option Pricing}
\label{sec:effDimAsianOption}
The pricing of an Asian option is based on the arithmetic average of
the
stock prices in a particular period of time. The payoff of the call
option at the end of period $T$ is
\begin{equation*}
    \text{payoff} = \max\left(\frac{1}{s} \sum_{j=1}^{s} S_{j} - K,0 \right),
\end{equation*}
 where $S_{j}$ is the stock price at time $t_j=jT/s, j=0,\ldots,s$, $t_s = T$ and
$K$ is the strike price. Based on the risk-neutral valuation
principle, the price of the call option at $t=0$ should be
\begin{equation*}
E(\de^{-rT} \text{payoff}),
\end{equation*}
where $E(\cdot)$ is the expected value of different path movements
of the stock price and $r$ is the risk-free interest rate. The stock
price movement is assumed to follow a geometric Brownian motion,
\begin{equation*}
S_{j} = S_{j-1} \de^{(r-0.5\sigma^2) T/s + \sigma \sqrt{T/s}
Z_j},
\end{equation*}
where $\sigma$ is the volatility, and
$Z_1, \ldots, Z_s$ are independent standard normal random variables.
See \cite{WS03b}  for the details.

\begin{table}
\begin{center}
\begin{tabular}{c|ccc|cc}
    & \multicolumn{3}{c|}{Wang \& Fang} & \multicolumn{2}{c}{Spline}
\\ \hline
$s$ & standard & BB & PCA & $d_{\trc}$ & $d_{\sup}$ \\ \hline
8 & 7 & 5 & 2 & 7 & 2 \\
16 & 14 & 7 & 2 & 14 & 2 \\
32 & 27 & 7 & 2 & 27 & 2 \\
\end{tabular}
\end{center}
\caption{The estimated effective dimensions of the Asian option
pricing
problem.}
\label{tab:AsiaOptEffDim}
\end{table}

Table~\ref{tab:AsiaOptEffDim} shows the results of the estimated
effective dimension in both senses for an Asian option pricing
problem with the following parameters: $N=2^{14}, S_0 = K = 100, \sigma=0.2,
r=0.1,$ and $T=1$  year using a Sobol point set. The
results in \cite{WangFang2003} are also duplicated. The column
``standard'' is the estimated truncation dimension of the original
pricing function by their method. The columns ``BB'' and ``PCA'' are
the estimated truncation dimensions when the Brownian bridge and principle component analysis dimension reduction methods are
applied, respectively. The $d_{\trc}$ of the spline are exactly the same
as the results in \cite{WangFang2003}. The $d_{\sup}$ are also the same
as the truncation dimension after applying PCA, which is the best
dimension reduction method in \cite{WangFang2003}. Moreover, the
superposition dimension $d_{\sup}=2$ is consistent with the analytical
results in \cite{WS03b}. The variances of the discounted payoff for the Asian
option pricing problem computed by the sample variance and the variance of the spline are shown in Table \ref{tab:AsiaOptVar},
which indicate that the accuracy of the spline for estimating effective dimension should be reasonably
good.

\begin{table}
\begin{center}
\begin{tabular}{c|ccc}
$s$ & $\sigma^2_{\mathrm{QMC}}(f)$ & $\sigma^2(Sf)$ &
$\sigma^2(Sf) / \sigma^2_{\mathrm{QMC}}(f)$ \\ \hline
8 & 70.3669 & 69.4148 & 0.9865 \\
16 & 71.6402 & 69.0217 &  0.9634 \\
32 & 72.6395 & 67.8101 &  0.9335 \\
\end{tabular}
\end{center}
\caption{The variances of the Asian option price problem estimated by
a QMC method and the variances of the splines.}
\label{tab:AsiaOptVar}
\end{table}

\section{Conclusion and Remarks}
A fast discrete Walsh transform over digital nets is derived in
this article. This transform can be applied to reduce the cost of 
calculating the discrete Walsh coefficients  from $\Order(N^2)$
to $\Order(N \log N)$. Because the kernel used is piecewise constant, the spline interpolant based on this kernel may not be particularly accurate.  However, the spline interpolant does facilitate the estimation of coarser quantities, such as the variances of the ANOVA effects and the effective dimension of the function. The numerical results for two different families of functions are compared with the analytical results and the results from the literature showing that the estimation of effective dimension via the spline is accurate and efficient.

Finally, we would like add two remarks. The function values in our interpolation method were sampled over a digital net, which was originally constructed for numerical quadrature. Sparse grids \cite{Bungartz} are another sampling scheme and approximation method that can be applied to both numerical quadrature and function interpolation. Both methods can be classified as \textit{a priori} grid optimization \cite{Bungartz}. However, the point sets are based on different criteria. Digital nets minimize the discrepancy of the point set \cite{N92}, while the sparse grid is selected based on the importance of the components in a tensor product of hierarchical function spaces. Sparse grids use $\Order(h^{-1} \, (\log h^{-1})^{s-1})$ points where $h$ is the mesh size in each dimension. This number is substantially smaller than $\Order(h^{-s})$, the number of points needed for an ordinary grid, but the number of points for a sparse grid increases exponentially with dimension for a given mesh size. On the other hand there exist digital nets with $N=p^m$ points for $m=0, 1, \ldots$, independent of the dimension, $s$.  Error bounds have been developed for sparse grid methods. An error analysis of the spline algorithm for digital nets is the object of future work.

The second remark is that the fast algorithms in this paper are due to the correspondence between the point set and the kernel. In \cite{ZengLeungHickernell2004}, a similar technique was applied to integration lattices and the fast Fourier transform (FFT). For computer experiments, one typically has control over where to sample the underlying function, but for some problems of approximating a function based on observational or field data, the locations of the points where the function is sampled can not be selected in advance. Such problems are suited to the Nonequidistant Fast Fourier Transform (NFFT) pioneered by Potts \cite{Potts}. The NFFT provides a fast and relatively accurate, but only approximate, discrete Fourier transform of the data.  The FFT applied to data sampled on an integration lattice and the Fast Walsh Transform (FWT) introduced here applied to data sampled on digital nets both provide fast discrete transforms with accuracy only limited by machine precision.

\bibliographystyle{amsplain}

\end{document}